\PassOptionsToPackage{hyperfootnotes=false}{hyperref}
\documentclass[final,nomarks]{dmtcs-episciences}
\usepackage[utf8]{inputenc}
\usepackage[T1]{fontenc}
\let\qed\relax
\usepackage{amsmath,amssymb,amsthm}
\usepackage{enumitem}
\usepackage{xcolor}
\usepackage{tikz}
\usetikzlibrary{calc,positioning}
\usepackage{subcaption}
\usepackage{chngcntr}
\usepackage{aliascnt}
\usepackage[capitalize]{cleveref}
\hypersetup{
	colorlinks=true,
	linkcolor=blue,
	citecolor=blue,
	urlcolor=black,
	pdftitle={Online Ramsey numbers of the claw versus cycles},
	pdfauthor={Hexuan Zhi and Yanbo Zhang},
	pdfsubject={Discrete Mathematics and Theoretical Computer Science, vol. 28:3, article 23 (2026)},
	pdfkeywords={Ramsey number, online Ramsey number, claw, cycles}
}

\newtheorem{thm}{Theorem}[section]
\newaliascnt{lemma}{thm}
\newtheorem{lemma}[lemma]{Lemma}
\aliascntresetthe{lemma}
\newaliascnt{Observation}{thm}
\newtheorem{Observation}[Observation]{Observation}
\aliascntresetthe{Observation}
\theoremstyle{definition}
\newtheorem{case}{Case}
\newtheorem{subcase}{Subcase}[case]
\newenvironment{customlemma}[1]{%
	\par\noindent\textbf{Lemma #1.}\itshape
}{\par}
\crefname{thm}{Theorem}{Theorems}
\Crefname{thm}{Theorem}{Theorems}
\crefname{lemma}{Lemma}{Lemmas}
\Crefname{lemma}{Lemma}{Lemmas}
\crefname{Observation}{Observation}{Observations}
\Crefname{Observation}{Observation}{Observations}
\crefname{case}{Case}{Cases}
\Crefname{case}{Case}{Cases}
\counterwithin{figure}{section}
\tikzset{
	global scale/.style={
		scale=#1,
		every node/.append style={scale=#1}
	}
}

\allowdisplaybreaks

\title[Online Ramsey numbers of the claw versus cycles]
{Online Ramsey numbers of the claw versus cycles}
\author[Hexuan Zhi and Yanbo Zhang]{%
	Hexuan Zhi
	\and Yanbo Zhang\thanks{Corresponding author.}}
\affiliation{%
	School of Mathematical Sciences, Hebei Normal University, Shijiazhuang, China\\
	Hebei Research Center of the Basic Discipline Pure Mathematics, Shijiazhuang, China}
\begin{document}
\publicationdata{vol. 28:3}{2026}{23}{10.46298/dmtcs.17271}{2026-01-12; 2026-01-12; 2026-10-27}{2026-10-29}
	\maketitle
% This is only used if you are compiling for a volume before vol 25
% \publicationdetails{VOL}{2015}{ISS}{NUM}{SUBM}
% This is the new form of collecting the data, starting with vol 25

%{1998-10-14; 1998-10-14; 2002-07-19; 2014-02-05; 2015-09-09; 2022-12-25}
%{2022-12-3}

\begin{abstract}
	The online Ramsey number $\tilde r(G,H)$ is defined via a Builder--Painter game on an empty graph with countably many vertices. In each round, Builder reveals an edge, which Painter immediately colors either red or blue. Builder wins once a red copy of $G$ or a blue copy of $H$ appears, and $\tilde r(G,H)$ is the minimum number of edges Builder must reveal to force a win.
	For a long cycle $C_\ell$, the online Ramsey numbers $\tilde r(G,C_\ell)$ are known only for a few specific choices of $G$.
	In particular, exact values were determined for $G=P_3$ by Cyman, Dzido, Lapinskas, and Lo (Electron.\ J.\ Combin., 2015), while asymptotically tight results were obtained when $G$ is an even cycle by Adamski, Bednarska-Bzd\k{e}ga, and Bla\v{z}ej (SIAM J.\ Discrete Math., 2024).
	In this paper, we consider the case where $G$ is the claw $K_{1,3}$ and determine the exact value of $\tilde r(K_{1,3},C_\ell)$.
	We show that $\tilde r(K_{1,3},C_\ell)=\left\lfloor \frac{3(\ell+1)}{2} \right\rfloor$ for all $\ell\ge13$.
\end{abstract}

\section{Introduction}
Online Ramsey numbers were introduced by Beck~\cite{Beck1993}, and the name originates from the work of Kurek and Ruci\'nski~\cite{Kurek2005}. Given two simple graphs $G$ and $H$, we consider a two-player game between \emph{Builder} and \emph{Painter}, called the \emph{$(G,H)$--online Ramsey game}. The game starts with an edgeless graph on countably many vertices. In each round, Builder selects two currently non-adjacent vertices and exposes the edge between them, and Painter immediately colors this edge either red or blue. Builder aims to force, in as few rounds as possible, the appearance of a red copy of $G$ or a blue copy of $H$, while Painter aims to delay this outcome for as long as possible. The \emph{online Ramsey number} $\tilde r(G,H)$ is the minimum number of rounds (equivalently, the minimum number of exposed edges) that Builder needs to guarantee a win.

For online Ramsey numbers involving complete graphs, Conlon~\cite{Conlon2010} established a general upper bound for $\tilde r(n,n)$. For lower bounds, Beck~\cite{Beck1993} gave an elegant argument showing that $\tilde r(n,n)\ge r(n,n)/2$. More recently, Conlon, Fox, Grinshpun, and He~\cite{Conlon2019} improved this bound by proving
\[
\tilde r(n,n)\ge 2^{(2-\sqrt2)n+O(1)}.
\] For small complete graphs,  Kurek and Ruci\'nski~\cite{Kurek2005} determined that $\tilde r(3,3)=8$. Subsequently, works by P. Pra\l{}at~\cite{Pralat2008cliques} and by Gordinowicz and P. Pra\l{}at~\cite{GordinowiczPralat2008} applied computer algorithms to determine exact values, establishing for instance that $\tilde r(3,4)=17$.

In contrast to the complete-graph case, the study of online Ramsey numbers for sparse graphs has been particularly active, especially for paths, stars, trees, and cycles; see, for instance,
\cite{Adamska2025,Adamski2024a,Adamski2024b,Adamski2024c,Bednarska2024,Cyman2015,Dybizbanski2020,Grytczuk2008,Latip2021,Mond2024,Pralat2008,Pralat2012,Song2025,Zhang2023}.
However, results concerning \emph{long} cycles in this direction remain rather limited, and exact values are known only in very few cases. This makes it particularly meaningful to determine precise online Ramsey numbers when the avoided graph is a long cycle.

Concerning long cycles, Adamski, Bednarska-Bzd\k{e}ga, and Bla\v{z}ej~\cite{Adamski2024c} proved that for all large $n$ and fixed $k$,
\[
\tilde r(C_k,C_n)=2n+\mathcal{O}(k)\quad \text{if $k$ is even,}
\]
while
\[
\tilde r(C_k,C_n)\le 3n+o(n)\quad \text{if $k$ is odd.}
\]
Cyman, Dzido, Lapinskas, and Lo~\cite{Cyman2015} determined that for $\ell\ge 5$,
\[
\tilde r(P_3,C_\ell)=\left\lceil \frac{5\ell}{4}\right\rceil.
\]
Beyond these, relatively few results are known for online Ramsey numbers involving long cycles. Song, Wang, and Zhang~\cite{Song2025} showed that for $\ell\ge 2$,
\[
\tilde r(K_{1,3},P_\ell)=\left\lfloor \frac{3\ell}{2}\right\rfloor.
\]
In this paper, we extend their work by determining the exact value of $\tilde r(K_{1,3},C_\ell)$. Our main result is the following.

\begin{thm}\label{thm:mainresult}
	For $\ell\ge 13$, we have
	\[
	\tilde{r}(K_{1,3}, C_\ell)= \left\lfloor\frac{3(\ell+1)}{2}\right\rfloor.
	\]
\end{thm}

To determine $\tilde r(K_{1,3},C_\ell)$, we give an explicit winning strategy for Builder. Starting from four small forced configurations, we apply case-specific routines that extend a controlled blue path structure until it can be closed to a blue $C_\ell$ within the claimed number of rounds.

We prove the lower bound in \cref{sec2}. In \cref{sec3--1} we introduce several technical lemmas and key structures, whose proofs are deferred to \cref{sec4}; the upper bound is established in \cref{sec3--2}.

\section{The Lower Bound for Theorem \ref{thm:mainresult}}\label{sec2}
\begin{proof}[of the lower bound in Theorem~\ref{thm:mainresult}]
	To establish the lower bound, we construct a strategy for Painter which guarantees that after
	$\left\lfloor \frac{3(\ell+1)}{2}\right\rfloor-1$ rounds, Builder still cannot force a blue copy of $C_\ell$.
	
	Throughout the game, Painter colors the exposed edge red unless doing so would create either a red $K_{1,3}$ or a red cycle $C_k$ with $3\le k\le \left\lfloor\frac{\ell+1}{2}\right\rfloor$.
	More precisely, let $R_i$ denote the graph formed by all red edges before round $i$, and let $e_i$ be the edge chosen by Builder in round $i$. If $R_i+e_i$ contains a copy of $K_{1,3}$ or a cycle of length at most $\left\lfloor\frac{\ell+1}{2}\right\rfloor$, then Painter colors $e_i$ blue; otherwise, Painter colors $e_i$ red.
	
	Let $G$, $R$, and $B$ be the graphs induced by all edges, all red edges, and all blue edges, respectively, at the end of round
	$\left\lfloor \frac{3(\ell+1)}{2}\right\rfloor-1$.
	We will show that $G$ contains neither a red $K_{1,3}$ nor a blue $C_\ell$.
	
	If $R$ contains at least $\left\lfloor\frac{\ell+1}{2}\right\rfloor+1$ edges, then $B$ contains at most
	\[
	\left\lfloor \frac{3(\ell+1)}{2}\right\rfloor-1-\left(\left\lfloor\frac{\ell+1}{2}\right\rfloor+1\right)
	\le \ell-1
	\]
	edges, and hence $B$ cannot contain $C_\ell$. Therefore, we may assume that $R$ has at most $\left\lfloor\frac{\ell+1}{2}\right\rfloor$ edges. By Painter's strategy, $R$ contains no cycles and has maximum degree at most $2$, and thus $R$ is a disjoint union of paths.
	
	Let $s$ be the number of connected components of $R$. Let $X$ be the set of vertices of degree $2$ in $R$, and for $1\le i\le s$ let $\{y_{2i-1},y_{2i}\}$ denote the two endpoints of the $i$th path component. Then $R$ has exactly $|X|+s$ edges, and this number is at most $\left\lfloor\frac{\ell+1}{2}\right\rfloor$. Hence
	\begin{equation}\label{sec2equ1}
		|X|+s\le \left\lfloor\frac{\ell+1}{2}\right\rfloor.
	\end{equation}
	
	Suppose that $G$ contains a blue copy of $C_\ell$. Since every vertex
	of $C_\ell$ has degree 2, the cycle can use at most $2|X|$ edges incident
	with vertices in $X$. Recall Painter's coloring rule: every blue edge
	is ``forced'', in the sense that it is colored blue only to avoid
	creating a red $K_{1,3}$ or a short red cycle. Equivalently, each blue
	edge is either incident with a vertex in $X$, or it is an edge joining
	the two endpoints of some path component, namely $y_{2i-1}y_{2i}$.
	Therefore, the total number of blue edges in this cycle is at most
	$2|X|+s$, and thus
	\begin{equation}\label{sec2equ2}
		2|X|+s \ge \ell.
	\end{equation}
	Combining this with~\cref{sec2equ1}, we further obtain
	\begin{equation}\label{sec2equ3}
		2|X|+s\le \ell+1-s.
	\end{equation}
	
	If $s\ge 2$, then by~\cref{sec2equ2,sec2equ3} we obtain
	\[
	\ell \le 2|X|+s \le \ell+1-s \le \ell-1,
	\]
	a contradiction.
	
	It remains to consider the case $s=1$. Then~\cref{sec2equ2,sec2equ3} imply that $2|X|+s=\ell$, and hence
	$|X|+s=\left\lfloor\frac{\ell+1}{2}\right\rfloor$. Consequently, $R$ is a single path on $\left\lfloor\frac{\ell+1}{2}\right\rfloor$ edges.
	If Builder joins the two endpoints of this path, then by Painter's rule the new edge is colored red, creating a red cycle of length
	$\left\lfloor\frac{\ell+1}{2}\right\rfloor+1$. As discussed above, this forces at most $\ell-1$ blue edges in total.
	If instead Builder never joins the two endpoints of $R$, then every
	blue edge must be incident with a degree-2 vertex of $R$, and therefore
	there are at most $2|X|=\ell-1$ blue edges in the assumed blue cycle.
In either case, the blue graph cannot contain a copy of $C_\ell$.
\end{proof}

\section{The Upper Bound for Theorem \ref{thm:mainresult}}\label{sec3}
\subsection{Overview and technical preparation}\label{sec3--1}
We first list all lemmas needed for the proof of the upper bound. To keep the argument transparent, we will use these lemmas as black boxes, postponing their complete proofs to~\cref{sec4}.

Fix $\ell\ge13$, and let $N=\left\lfloor \frac{3(\ell+1)}{2}\right\rfloor$. Assuming that no red $K_{1,3}$ ever appears during the game, it suffices to show that Builder can force a blue copy of $C_\ell$ within $N$ rounds.

We call a vertex \emph{good} if it is incident with at least one red edge, and \emph{better} if it is incident with two red edges. A blue path whose two endpoints are a good vertex and a better vertex is called a \emph{better path}; if both endpoints are better vertices, we call it a \emph{best path}.

\begin{lemma}\label{lem:main1}
	Let $m\ge 2$ and $\ell-m\ge 2$. Suppose that Builder has constructed
	a blue $P_m$ within $k$ rounds and that one of the following holds:
	\begin{itemize}
		\item $P_m$ is a better path and $2k\le 3m-3$;
		\item $P_m$ is a best path and $2k\le 3m-1$.
	\end{itemize}
	Then Builder can force a blue $C_\ell$ within a total of
	$N$ rounds.
	Moreover, if $\ell-m$ is odd, then the condition $2k\le 3m-3$
	in the first case can be relaxed to $2k\le 3m-2$.
\end{lemma}

\begin{lemma}\label{lem:main2}
	Suppose Builder has constructed a blue $P_{\ell-1}$ within $k$ rounds.
	\begin{itemize}
		\item If $P_{\ell-1}$ is a better path, then Builder can force a blue $C_\ell$ within a total of $k+3$ rounds.
		\item If $P_{\ell-1}$ is a best path, then Builder can force a blue $C_\ell$ within a total of $k+2$ rounds.
	\end{itemize}
\end{lemma}

\begin{lemma}\label{lem:betterpl-2}
	For $\ell\ge 4$, if Builder has already constructed a better $P_{\ell-2}$, then Builder can create a blue $C_{\ell}$ within five additional rounds.
\end{lemma}

Next we introduce the notions of \emph{Type~1}, \emph{Type~2}, \emph{Type~3}, and the \emph{wavy path}. We will discuss properties of these configurations, which in certain situations will serve as building blocks for constructing a blue cycle.

\noindent\textbf{Types.}
Let $m\ge 3$. We define three classes of configurations, called \emph{Type~1}, \emph{Type~2}, and \emph{Type~3}. Suppose Builder has already constructed a blue path $P_m$ with endpoints $v_1$ and $v_m$; we refer to this as the \emph{blue main path}. Starting from this blue main path, we attach one of three local gadgets at the endpoint $v_m$, thereby obtaining the configurations of Type~1, Type~2, and Type~3. See \cref{fig:type1,fig:type2,fig:type3} for illustrations.
\begin{itemize}
	\item \textbf{Type 1:} Attach a triangle at $v_m$ as follows: the edges $v_mx$ and $xy$ are blue, the edge $v_my$ is red, and in addition there is a red edge $yz$.
	\item \textbf{Type 2:} Add two red edges incident to $v_m$, namely $v_mx$ and $v_my$.
	\item \textbf{Type 3:} Attach a triangle at $v_m$ such that $v_mx$ is blue, while $v_my$ and $xy$ are red.
\end{itemize}

\begin{figure}[ht]
	\centering
	\definecolor{darkgreen}{rgb}{0.0, 0.5, 0.0} 
	\begin{subfigure}{0.3\textwidth}
		\centering
		\begin{tikzpicture}[scale=1.3, 
			vertex/.style={circle, draw, fill=white, inner sep=1.5pt},
			every label/.style={below, scale=1, black}]
			
			\node[vertex, label=below:$v_1$] (v1) at (0, 0) {};
			\node[vertex, label=below:$v_{m}$] (vm) at (2, 0) {};
			\node[vertex, label=below:$x$] (x) at (2.25, 0.5) {};
			\node[vertex, label=below:$y$] (y) at (2.5, 0) {};
			\node[vertex, label=below:$z$] (z) at (3, 0) {};
			\node[ label=below:$P_m$] () at (1, 0.75) {};
			
			\draw[line width=0.5mm, blue, bend left] (v1) to (vm);
			\draw[line width=0.5mm, red] (vm) -- (y);
			\draw[line width=0.5mm, blue] (vm) -- (x);
			\draw[line width=0.5mm, blue] (y) -- (x);
			\draw[line width=0.5mm, red ] (y) -- (z);
		\end{tikzpicture}
		\caption{Type 1}
		\label{fig:type1}
	\end{subfigure}
	\begin{subfigure}{0.3\textwidth} 
		\centering
		\begin{tikzpicture}[scale=1.3, 
			vertex/.style={circle, draw, fill=white, inner sep=1.5pt},
			every label/.style={below, scale=1, black}]
			
			\node[vertex, label=below:$v_1$] (v1) at (0, 0) {};
			\node[vertex, label=below:$v_m$] (vm) at (2, 0) {};
			\node[vertex, label=below:$x$] (x) at (2.25, 0.5) {};
			\node[vertex, label=below:$y$] (y) at (2.5, 0) {};
			\node[ label=below:$P_m$] () at (1, 0.75) {};
			
			\draw[line width=0.5mm, blue, bend left] (v1) to (vm);
			\draw[line width=0.5mm, red] (vm) -- (y);
			\draw[line width=0.5mm, red] (vm) -- (x);			
		\end{tikzpicture}
		\caption{Type 2}
		\label{fig:type2}
	\end{subfigure}
	\begin{subfigure}{0.3\textwidth} 
		\centering
		\begin{tikzpicture}[scale=1.3, 
			vertex/.style={circle, draw, fill=white, inner sep=1.5pt},
			every label/.style={below, scale=1, black}]
			
			\node[vertex, label=below:$v_1$] (v1) at (0, 0) {};
			\node[vertex, label=below:$v_m$] (vm) at (2, 0) {};
			\node[vertex, label=below:$x$] (x) at (2.25, 0.5) {};
			\node[vertex, label=below:$y$] (y) at (2.5, 0) {};
			\node[ label=below:$P_m$] () at (1, 0.75) {};
			
			\draw[line width=0.5mm, blue, bend left] (v1) to (vm);
			\draw[line width=0.5mm, red] (vm) -- (y);
			\draw[line width=0.5mm, blue] (vm) -- (x);
			\draw[line width=0.5mm, red] (x) -- (y);
			
		\end{tikzpicture}
		\caption{Type 3}
		\label{fig:type3}
	\end{subfigure}
	\caption{Configurations of Types~1, 2, and~3.}
	\label{fig:type}
\end{figure}

\begin{lemma}\label{lem:type}
	Assume that $k$ and $m$ satisfy $3m-2k\ge 5$ and $\ell-m\ge 6$.
	If Builder constructs a Type~1 configuration within $k+4$ rounds, or a Type~2 configuration within $k+2$ rounds, or a Type~3 configuration within $k+3$ rounds, then in each case Builder can force a blue $C_\ell$ within a total of $N$ rounds.
\end{lemma}

The lemma above shows that once Builder has produced a \emph{Type~1}, \emph{Type~2}, or \emph{Type~3} configuration, it suffices to focus on the blue main path $P_m$. If the inequalities $3m-2k\ge 5$ and $\ell-m\ge 6$ hold, then Builder can force a blue $C_\ell$ in at most $N$ rounds.

\medskip
\noindent\textbf{Wavy path.}
Let $i\ge 2$ and suppose that $i\ge 2j+1$. Consider a blue path $P_{i+j}$. If its edge set contains $j\ge 0$ pairwise vertex-disjoint subpaths of length~$2$ (i.e., copies of $P_3$), and moreover the two endpoints of each such $P_3$ are joined by a red edge, then we call the graph formed by this blue path together with these red edges a \emph{wavy path}, denoted by $P(i,j)$. In $P(i,j)$, each $P_3$ together with the red edge joining its endpoints forms a triangle; we call such a triangle a \emph{wavy triangle}. The unique vertex of a wavy triangle that is not incident with the red edge is called a \emph{wavy point}.

For later convenience, we adopt a standardized labeling convention. In $P(i,j)$, we label all wavy points (if any) consecutively by $u_1,\ldots,u_j$, and label the remaining vertices along the underlying path order by $v_1,\ldots,v_i$.

A schematic illustration of a general wavy path is given in~\cref{fig:wavypath}. In that figure, $P_m$ denotes a blue path with endpoints $v_1$ and $v_m$, and $P_n$ denotes a blue path with endpoints $v_{m+1}$ and $v_{m+n}$, where $m,n\ge 2$. We also present concrete examples of wavy paths, namely $P(4,1)$ and $P(4,0)$. It is easy to check that, up to isomorphism, there are exactly two possibilities for $P(4,1)$ and exactly one for $P(4,0)$; see~\cref{fig:P410}.

Next we record several basic properties of wavy paths. A straightforward count shows that $P(i,j)$ has exactly $i+j$ vertices and $i+2j-1$ edges: among them, $i+j-1$ edges come from the blue path $P_{i+j}$, and the remaining $j$ edges are red. By definition, any two wavy triangles in $P(i,j)$ are vertex-disjoint, and each red edge belongs to a distinct wavy triangle. Consequently, $P(i,j)$ contains no red copy of $P_3$. When $j=0$, the graph $P(i,j)$ is simply a blue $P_i$.

\begin{figure}[htp]
	\centering
	\begin{tikzpicture}[
		node/.style={circle, draw, fill=white, inner sep=1.5pt},
		bend edge/.style={blue, thick, bend left=20}
		]
		
		\coordinate (1) at (0,0);
		\coordinate (label1) at (0.75,0);
		\coordinate (2) at (1.5,0);
		\coordinate (3) at (2,1);
		\coordinate (4) at (2.5,0);
		\coordinate (label2) at (3.25,0);
		\coordinate (5) at (4,0);
		\coordinate (6) at (4.5,1);
		\coordinate (7) at (5,0);
		\coordinate (71) at (5.25,0);
		\coordinate (8) at (7,0);
		\coordinate (label3) at (7.75,0);
		\coordinate (9) at (8.5,0);
		\coordinate (10) at (9,1);
		\coordinate (11) at (9.5,0);
		\coordinate (12) at (11,0);
		
		\draw[blue, very thick] 
		(1) -- (2) -- (3) -- (4)-- (5)--(6)--(7);
		\draw[blue, very thick] 
		(8) -- (9) -- (10) -- (11);
		\draw[red,very thick] (2) -- (4) ;
		\draw[red,very thick] (5) --(7) ;
		\draw[red,very thick] (9) --(11) ;
		
		\begin{scope}[line width=1pt, black, dotted]
			\draw (7) to (8);
			\draw (11) to (12);
		\end{scope}

		\node[node] at (1) {};
		\node[node] at (2) {};
		\node[node] at (3) {};
		\node[node] at (4) {};
		\node[node] at (5) {};
		\node[node] at (6) {};
		\node[node] at (7) {};
		\node[node] at (8) {};
		\node[node] at (9) {};
		\node[node] at (10) {};
		\node[node] at (11) {};
		\node[node] at (12) {};

		\node[below=1mm of 1, label] {$v_1$};
		\node[above=1mm of label1, label] {$P_m$};
		\node[below=1mm of 2, label] {$v_m$};
		\node[above=1mm of 3, label] {$u_1$};
		\node[below=1mm of 4, label] {$v_{m+1}$};
		\node[above=1mm of label2, label] {$P_n$};
		\node[below=1mm of 5, label] {$v_{m+n}$};
		\node[above=1mm of 6, label] {$u_2$};
		\node[below=1mm of 71, label] {$v_{m+n+1}$};

	\end{tikzpicture}
	\caption{Wavy Path}
	\label{fig:wavypath}
\end{figure}

\begin{figure}[ht]
	\centering
	\begin{subfigure}{0.32\textwidth}
		\centering
		\begin{tikzpicture}[
			node/.style={circle, draw, fill=white, inner sep=1.5pt},
			bend edge/.style={blue, thick, bend left=20}
			]
			
			\coordinate (1) at (0,0);
			\coordinate (2) at (1,0);
			\coordinate (3) at (1.5,1);
			\coordinate (4) at (2,0);
			\coordinate (5) at (3,0);
			
			\draw[blue, very thick] (1) -- (2) -- (3) -- (4) -- (5);
			
			\node[below=3pt] at (1) {$v_1$};
			\node[below=3pt] at (2) {$v_2$};
			\node[above=3pt] at (3) {$u_1$}; 
			\node[below=3pt] at (4) {$v_3$};
			\node[below=3pt] at (5) {$v_4$};
			
			\draw[red, very thick] (2) -- (4);
			
			\node[node] at (1) {};
			\node[node] at (2) {};
			\node[node] at (3) {};
			\node[node] at (4) {};
			\node[node] at (5) {};
		\end{tikzpicture}
		\caption{$P(4,1)$}
	\end{subfigure}
	\begin{subfigure}{0.33\textwidth}
		\centering
		\begin{tikzpicture}[
			node/.style={circle, draw, fill=white, inner sep=1.5pt},
			bend edge/.style={blue, thick, bend left=20}
			]
			
			\coordinate (1) at (0,0);
			\coordinate (2) at (1,0);
			\coordinate (3) at (2,0);
			\coordinate (4) at (2.5,1);
			\coordinate (5) at (3,0);
			
			\draw[blue, very thick] (1) -- (2) -- (3) -- (4) -- (5);
			
			\node[below=3pt] at (1) {$v_1$};
			\node[below=3pt] at (2) {$v_2$};
			\node[below=3pt] at (3) {$v_3$}; 
			\node[above=3pt] at (4) {$u_1$};
			\node[below=3pt] at (5) {$v_4$};
			
			\draw[red, very thick] (3) -- (5);
			
			\node[node] at (1) {};
			\node[node] at (2) {};
			\node[node] at (3) {};
			\node[node] at (4) {};
			\node[node] at (5) {};
		\end{tikzpicture}
		\caption{$P(4,1)$}
	\end{subfigure}
	\begin{subfigure}{0.33\textwidth}
		\centering
		\begin{tikzpicture}[
			node/.style={circle, draw, fill=white, inner sep=1.5pt},
			bend edge/.style={blue, thick, bend left=20}
			]
			
			\coordinate (1) at (0,0);
			\coordinate (2) at (1,0);
			\coordinate (3) at (2,0);
			\coordinate (4) at (3,0);

			\draw[blue, very thick] (1) -- (2) -- (3) -- (4);
			
			\node[below=3pt] at (1) {$v_1$};
			\node[below=3pt] at (2) {$v_2$};
			\node[below=3pt] at (3) {$v_3$}; 
			\node[below=3pt] at (4) {$v_4$};
			
			\node[node] at (1) {};
			\node[node] at (2) {};
			\node[node] at (3) {};
			\node[node] at (4) {};
			
		\end{tikzpicture}
		\caption{$P(4,0)$}
	\end{subfigure}
	\caption{The two drawings of $P(4,1)$ and the drawing of $P(4,0)$.
	}
	\label{fig:P410}
\end{figure}

\begin{lemma}\label{lem:G1}
	Let $i+j\le \ell-3$ and $i\ge 4$. Suppose Builder has constructed a copy of $P(i,j)$ within $i+2j-1$ rounds, and moreover the edge $v_1v_2$ is blue. Then Builder can force a blue $C_\ell$ within a total of $N$ rounds.
\end{lemma}

Note that the quantity $i+2j-1$ in the lemma is exactly the number of edges of $P(i,j)$. Thus, when the graph built by Builder is precisely $P(i,j)$ and $v_1v_2$ is blue, Builder can force a blue $C_\ell$ within a total of $N$ rounds.

\subsection{Proof of the upper bound}\label{sec3--2}
\begin{Observation}\label{obs:first}
	Builder either constructs $G_1$ within two rounds, or, after one additional round, constructs one of $G_2$, $G_3$, or $G_4$  (see \cref{fig:first}).
\end{Observation}

\begin{figure}[ht]
	\centering
	
	\begin{subfigure}[b]{0.24\textwidth}
		\centering
		\begin{tikzpicture}[scale=1.9,
			vertex/.style={circle, draw, fill=white, inner sep=1.5pt},
			every label/.style={scale=1, black}]
			\node[vertex, label=below:$v_1$] (v1) at (0, 0) {};
			\node[vertex, label=below:$v_2$] (v2) at (0.5, 0) {};
			\node[vertex, label=below:$v_3$] (v3) at (1, 0) {};
			\draw[line width=0.5mm, blue] (v1) -- (v2);
			\draw[line width=0.5mm, blue] (v2) -- (v3);
		\end{tikzpicture}
		\caption{$G_1$}
		\label{G1}
	\end{subfigure}\hfill
	\begin{subfigure}[b]{0.24\textwidth}
		\centering
		\begin{tikzpicture}[scale=1.9,
			vertex/.style={circle, draw, fill=white, inner sep=1.5pt},
			every label/.style={scale=1, black}]
			\node[vertex, label=below:$v_2$] (v2) at (0, 0) {};
			\node[vertex, label=below:$v_3$] (v3) at (1, 0) {};
			\node[vertex, label=left:$u_2$]   (u2) at (0.5, 0.5) {};
			\draw[line width=0.5mm, red]  (v2) -- (v3);
			\draw[line width=0.5mm, blue] (v2) -- (u2);
			\draw[line width=0.5mm, blue] (v3) -- (u2);
		\end{tikzpicture}
		\caption{$G_2$}
		\label{G2}
	\end{subfigure}\hfill
	\begin{subfigure}[b]{0.24\textwidth}
		\centering
		\begin{tikzpicture}[scale=1.9,
			vertex/.style={circle, draw, fill=white, inner sep=1.5pt},
			every label/.style={scale=1, black}]
			\node[vertex, label=below:$x$] (x) at (0, 0) {};
			\node[vertex, label=below:$y$] (y) at (1, 0) {};
			\node[vertex, label=left:$v_0$]   (v0) at (0.5, 0.5) {};
			\draw[line width=0.5mm, blue] (x) -- (y);
			\draw[line width=0.5mm, red]  (x) -- (v0);
			\draw[line width=0.5mm, red]  (y) -- (v0);
		\end{tikzpicture}
		\caption{$G_3$}
		\label{G3}
	\end{subfigure}\hfill
	\begin{subfigure}[b]{0.24\textwidth}
		\centering
		\begin{tikzpicture}[scale=1.9,
			vertex/.style={circle, draw, fill=white, inner sep=1.5pt},
			every label/.style={scale=1, black}]
			\node[vertex, label=below:$v_1$] (v1) at (0, 0) {};
			\node[vertex, label=below:$v_2$] (v2) at (1, 0) {};
			\node[vertex, label=left:$v_3$]   (v3) at (0.5, 0.5) {};
			\draw[line width=0.5mm, red] (v1) -- (v2);
			\draw[line width=0.5mm, red] (v1) -- (v3);
			\draw[line width=0.5mm, red] (v2) -- (v3);
		\end{tikzpicture}
		\caption{$G_4$}
		\label{G4}
	\end{subfigure}
	\caption{}
	\label{fig:first}
\end{figure}

\begin{proof}
	Builder first creates a $P_3$. If both edges of this $P_3$ are blue, then we obtain $G_1$. If the $P_3$ contains one red edge and one blue edge, then Builder joins the two endpoints of the $P_3$, thereby exposing an edge $e_1$. If $e_1$ is blue, we obtain $G_2$; otherwise, we obtain $G_3$. If both edges of the $P_3$ are red, then Builder again joins its two endpoints, exposing an edge $e_2$. If $e_2$ is blue, we obtain $G_3$; otherwise, we obtain $G_4$.
\end{proof}

Next, we design separate algorithms for the cases \cref{G1,G2,G3} that extend the obtained configuration to a blue $C_\ell$; the case \cref{G4} is easier to handle. Our overall strategy is as follows. By \cref{obs:first}, Builder first forces one of $G_1$, $G_2$, $G_3$, and $G_4$; we then apply the corresponding algorithm to the resulting graph to complete the proof.

\begin{thm}\label{thm:G1}
	If Builder constructs $G_1$ within the first two rounds of the $(K_{1,3},C_\ell)$--online Ramsey game, then Builder can force a blue $C_\ell$ within a total of $N$ rounds.
\end{thm}
\begin{proof}
	Builder first introduces a new vertex $v_4$ and exposes the edge $v_3v_4$. If this edge is blue, then Builder has obtained $P(4,0)$. If $v_3v_4$ is red, then Builder introduces a new vertex $u_1$ and exposes the edge $v_3u_1$. If $v_3u_1$ is red, then a Type~2 configuration is formed. If instead $v_3u_1$ is blue, Builder next exposes the edge $v_4u_1$. If $v_4u_1$ is blue, then Builder has obtained $P(4,1)$; if $v_4u_1$ is red, then a Type~3 configuration is formed.
	
	In both the Type~2 and Type~3 outcomes above, the blue main path is precisely the blue $P_3$ constructed within the first two rounds, and thus the desired conclusion follows from \cref{lem:type}. In the remaining outcomes $P(4,0)$ and $P(4,1)$, the conclusion follows from \cref{lem:G1}.
\end{proof}

\begin{thm}\label{thm:G2}
	If Builder constructs $G_2$ within the first three rounds of the $(K_{1,3},C_\ell)$--online Ramsey game, then Builder can force a blue $C_\ell$ within a total of $N$ rounds.
\end{thm}

\begin{proof}
	We use the vertex labeling of $G_2$ shown in~\cref{G2}. Builder now introduces two new vertices $v_1$ and $v_4$, and exposes the edges $v_1v_2$ and $v_3v_4$. We consider the following three cases.
	\setcounter{case}{0}
	\begin{case}
		Both $v_1v_2$ and $v_3v_4$ are blue.
	\end{case}
	In this case, within five rounds Builder has constructed a copy of $P(4,1)$ with $v_1v_2$ blue, and therefore the conclusion follows from~\cref{lem:G1}.
	
	\begin{case}
		Exactly one of $v_1v_2$ and $v_3v_4$ is red and the other is blue. Without loss of generality, assume that $v_1v_2$ is red and $v_3v_4$ is blue.
	\end{case}
	Builder exposes the edge $v_1v_4$. If this edge is blue, then $v_1v_4v_3u_2v_2$ forms a better blue $P_5$ constructed within $6$ rounds, and we can apply~\cref{lem:main1} to finish the proof.
	If instead $v_1v_4$ is red, then starting with $i=4$ Builder repeatedly performs the following step: introduce a new vertex $v_{i+1}$ and expose the edge $v_iv_{i+1}$. This process stops as soon as either $v_iv_{i+1}$ is blue or the edge $v_{\lfloor \ell/2\rfloor}v_{\lfloor \ell/2\rfloor+1}$ is exposed; otherwise, we increment $i$ and continue from $v_{i+1}$. After the process terminates, we refer to~\cref{H1}. Depending on whether the final edge $v_iv_{i+1}$ is red or blue, we proceed with separate analyses.
	
	\begin{figure}[htp]
		\centering
		\begin{tikzpicture}[global scale=1.2]
			\begin{scope}[line width=1pt, blue]
				\draw (2,1) to (2.5,2);
				\draw (2.5,2) to (3,1);
				\draw (3,1) to (4,1);
				\draw[red, line width=1pt] (7,1.015) -- (8,1.015);
				\draw[blue, line width=1pt] (7,0.995) -- (8,0.995);
			\end{scope}      	  		
			
			\begin{scope}[line width=1pt, red]
				\draw (1,1) to (2,1);
				\draw(1,1) to[bend right=30] (4,1);
				\draw (2,1) to (3,1);
				\draw (4,1) to (5,1);
				\draw (6,1) to (7,1);
			\end{scope}
			\begin{scope}[line width=1pt, red, dotted]
				\draw (4.5,1) to (5,1);
				\draw (5,1) to (6,1);
			\end{scope}
			\filldraw[fill=white] (1,1) circle (2pt);
			\filldraw[fill=white](2,1) circle (2pt);
			\filldraw[fill=white] (3,1) circle (2pt);
			\filldraw[fill=white] (4,1) circle (2pt);
			\filldraw[fill=white] (5,1) circle (2pt);
			\filldraw[fill=white]
			(6,1) circle (2pt);
			\filldraw[fill=white] (7,1) circle (2pt);
			\filldraw[fill=white](8,1) circle (2pt);
			\filldraw[fill=white] (2.5,2)circle(2pt);
			
			\node at (1,0.7) {$v_1$};
			\node at (2,0.7) {$v_2$};
			\node at (3,0.7) {$v_3$};
			\node at (4,0.7) {$v_4$};
			\node at (7,0.7) {$v_i$};
			\node at (2.5,2.3) {$u_2$};
			\node at (8,0.7) {$v_{i+1}$};		
		\end{tikzpicture}
		\caption{}
		\label{H1}
	\end{figure}
	
	\setcounter{subcase}{0}
	
	\begin{subcase}
		$i=4$ and $v_iv_{i+1}$ is blue.
	\end{subcase}
	Builder exposes the edge $v_1v_{i+1}$, which must be blue. Hence $v_1v_5v_4v_3u_2v_2$ forms a best blue $P_6$ constructed within $8$ rounds, and we can apply~\cref{lem:main1} to complete the proof.
	
	\begin{subcase}
		$5\le i\le \left\lfloor \ell/2\right\rfloor$ and $v_iv_{i+1}$ is blue.
	\end{subcase}
	Builder exposes the edges $v_1v_i$ and $v_{i-1}v_{i+1}$; both must be blue. If $i=5$, then within $10$ rounds Builder has constructed a best blue $P_7$, and we can apply~\cref{lem:main1}. If $6\le i\le \left\lfloor \ell/2\right\rfloor$, then Builder introduces new vertices $u_3,\ldots,u_{i-3}$ and, for each $6\le j\le i$, exposes the edges $v_{j-2}u_{j-3}$ and $u_{j-3}v_{j-1}$, all of which must be blue. Consequently, within $3i-5$ rounds Builder constructs a best blue $P_{2i-3}$, and again~\cref{lem:main1} completes the proof.
	
	\begin{subcase}
		The edge $v_{\lfloor \ell/2\rfloor}v_{\lfloor \ell/2\rfloor+1}$ is red.
	\end{subcase}
	Builder introduces new vertices $u_1,u_3,\ldots,u_{i-2}$. Next, Builder exposes the edges $v_1u_1$ and $u_1v_2$, and for each $5\le j\le i$ exposes the edges $v_{j-1}u_{j-2}$ and $u_{j-2}v_j$; all these edges must be blue.
	
	If $\ell$ is even, then Builder exposes the edge $v_1v_{i+1}$, which must be blue. Thus, within $3\ell/2-2$ rounds Builder has constructed a better blue $P_{\ell-1}$, and the claim follows from~\cref{lem:main2}.
	
	If $\ell$ is odd, then $i=(\ell-1)/2$. Builder introduces a new vertex $v_{i+2}$ and exposes the edge $v_{i+1}v_{i+2}$. If $v_{i+1}v_{i+2}$ is blue, then Builder exposes $v_1v_{i+2}$, which must be blue, and hence within $(3\ell-1)/2-2$ rounds Builder has constructed a better blue $P_{\ell-1}$. Applying~\cref{lem:main2} completes the proof.
	If instead $v_{i+1}v_{i+2}$ is red, then Builder introduces new vertices $u_{i-1}$ and $u_i$, and exposes the edges
	$v_iu_{i-1}$, $u_{i-1}v_{i+1}$, $v_{i+1}u_i$, and $u_iv_1$, all of which must be blue. Consequently, within $(3\ell+1)/2$ rounds Builder constructs a blue $C_\ell$.

	\begin{case}
		Both $v_1v_2$ and $v_3v_4$ are red.
	\end{case}
	We extend the red path $v_1v_2v_3v_4$ starting from $v_4$ as follows. For each $i\ge 4$, as long as the edge $v_{i-1}v_i$ is red, Builder connects $v_i$ to two new vertices, namely $u_{i-1}$ and $v_{i+1}$. If both newly exposed edges $v_iu_{i-1}$ and $v_iv_{i+1}$ are blue, then we stop the extension. Otherwise, without loss of generality, assume that $v_iu_{i-1}$ is blue while $v_iv_{i+1}$ is red; in this case Builder continues the same procedure from $v_{i+1}$, thereby extending the red path. The process terminates either when the red path reaches length $\left\lfloor \ell/2\right\rfloor$, or when there exists some $t$ with $4\le t\le \left\lfloor \ell/2\right\rfloor$ such that both $v_tu_{t-1}$ and $v_tv_{t+1}$ are blue. After the process terminates, we refer to~\cref{H2}. According to whether the edge $v_tv_{t+1}$ is red or blue, we proceed with separate analyses.

	\begin{figure}[htp]
		\centering
		\begin{tikzpicture}[global scale=1.2]
			\begin{scope}[line width=1pt, blue]
				\draw[red, line width=1pt] (7,1.015) -- (8,1.015);
				\draw[blue, line width=1pt] (7,0.985) -- (8,0.985);
				\draw (2,1) to (2.5,2);
				\draw (4,1) to (4,2);
				\draw (7,1) to (7,2);
				\draw (2.5,2) to (3,1);
			\end{scope}      	  		
			
			\begin{scope}[line width=1pt, red]
				\draw (1,1) to (2,1);
				\draw (2,1) to (3,1);
				\draw (3,1) to (4,1);
				\draw (4,1) to (5,1);
				\draw (6,1) to (7,1);
			\end{scope}
			\begin{scope}[line width=1pt, red, dotted]
				\draw (4.5,1) to (5,1);
				\draw (5,1) to (6,1);
			\end{scope}
			\filldraw[fill=white](1,1) circle (2pt);%v_1
			\filldraw[fill=white] (2,1) circle (2pt);%v_2
			\filldraw[fill=white] (3,1) circle (2pt);%v_3
			\filldraw[fill=white] (4,1) circle (2pt);%v_4
			\filldraw[fill=white]
			(5,1) circle (2pt);
			\filldraw[fill=white]
			(6,1) circle (2pt);
			\filldraw[fill=white] (7,1) circle (2pt);%v_t
			\filldraw[fill=white] (8,1) circle (2pt);%v_{t+1}
			\filldraw[fill=white] (2.5,2) circle (2pt);%u_2
			\filldraw[fill=white] (4,2) circle (2pt);%u_4
			\filldraw[fill=white] (7,2) circle (2pt);%u_t-1
			\node at (1,0.7) {$v_1$};
			\node at (2,0.7) {$v_2$};
			\node at (3,0.7) {$v_3$};
			\node at (4,0.7) {$v_4$};
			\node at (7,0.7) {$v_t$};
			\node at (2.5,2.3) {$u_2$};
			\node at (4,2.3) {$u_3$};
			\node at (7,2.3) {$u_{t-1}$};
			\node at (8,0.7) {$v_{t+1}$};		
		\end{tikzpicture}
		\caption{}
		\label{H2}
	\end{figure}
	
	\setcounter{subcase}{0}
	
	\begin{subcase}
		There exists some $t$ with $4\le t\le \left\lfloor \ell/2\right\rfloor$ such that both $v_tu_{t-1}$ and $v_tv_{t+1}$ are blue.
	\end{subcase}
	Builder first exposes the edge $v_1v_{t+1}$. If this edge is blue, then for each $3\le x\le t-1$ Builder exposes the edge $v_xu_x$, which must be blue. Hence, within $3t-3$ rounds Builder constructs a better $P_{2t-1}$. If instead $v_1v_{t+1}$ is red, then Builder first exposes the edges $v_xu_x$ for all $4\le x\le t-1$ (if any), which must be blue, and then exposes the edges $v_1u_3$ and $v_3v_{t+1}$, both of which must be blue. Consequently, within $3t-2$ rounds Builder constructs a best $P_{2t-1}$.
	
	For either of the two blue paths obtained above, if $4\le t\le \left\lfloor \ell/2\right\rfloor-1$, or if $t=\left\lfloor \ell/2\right\rfloor$ and $\ell$ is odd, then we may apply~\cref{lem:main1} to finish the proof. If $t=\left\lfloor \ell/2\right\rfloor$ and $\ell$ is even, then~\cref{lem:main2} yields the desired conclusion.
	
	\begin{subcase}
		The red path reaches length $\left\lfloor \ell/2\right\rfloor$.
	\end{subcase}
	In this case $t=\left\lfloor \ell/2\right\rfloor$. Builder first exposes the edges $v_xu_x$ for all $3\le x\le t-1$, all of which are blue, and then exposes the edge $v_1v_{t+1}$.
	
	If $v_1v_{t+1}$ is blue, then Builder next exposes the edge $v_2v_{t+1}$, which must also be blue. Thus, within $3t-2$ rounds Builder constructs a better $P_{2t-1}$. If $\ell$ is even, then applying~\cref{lem:main2} yields a blue $C_\ell$ within a total of $3\ell/2+1$ rounds. If $\ell$ is odd, then $t=\left\lfloor \ell/2\right\rfloor=(\ell-1)/2$, so Builder has constructed a better $P_{\ell-2}$ within $(3\ell-3)/2-2$ rounds; applying~\cref{lem:betterpl-2} then produces a blue $C_\ell$ within a total of $N$ rounds.
	
	If $v_1v_{t+1}$ is red, then Builder proceeds according to the parity of $\ell$. If $\ell$ is even, Builder introduces a new vertex $u_1$ and exposes the edges $v_1u_1$ and $u_1v_2$, both of which must be blue. Hence Builder constructs a best $P_{\ell-1}$ within $3t-1=3\ell/2-1$ rounds, and the conclusion follows from~\cref{lem:main2}. If $\ell$ is odd, Builder introduces two new vertices $u$ and $w$, and sequentially exposes the edges $v_2u$, $uv_{t+1}$, $v_{t+1}w$, $wv_1$, and $v_1v_t$, all of which must be blue. Consequently, within $3t+2\le N$ rounds Builder obtains a blue cycle $C_\ell$, namely $v_2u_2\cdots v_{t-1}u_{t-1}v_tv_1wv_{t+1}uv_2$.
\end{proof}

To handle the case \cref{G3}, we introduce a technique for extending blue paths, which we call the \emph{Blue Path Expansion Algorithm}. Starting from a blue path, the algorithm produces a graph on
$\alpha\in\{\ell-3,\ell-2,\ell-1\}$ vertices, from which we will then construct a blue $C_\ell$.
\par\medskip
\noindent\textbf{Blue Path Expansion Algorithm.}
\begin{description}
	
	\item[\textbf{Start.}]
	Given a graph $G$ built by Builder, let $x$ and $y$ be the endpoints of a longest blue path in $G$.
	
	\item[\textbf{Phase 1.}]
	Builder introduces a new vertex $v_1$ and exposes the edge $yv_1$.
	\begin{itemize}
		\item If $yv_1$ is blue, then we call Phase~1 an instance of \emph{\textbf{Process~1}}. In this case the resulting graph contains a longer blue path, and we proceed to \textbf{End~1}.
		\item If $yv_1$ is red, we move to \textbf{Phase~2}. Moreover, the action of introducing $v_1$ and exposing the red edge $yv_1$ is regarded as \textbf{Step~1} of Phase~2.
	\end{itemize}
	
	\item[\textbf{Phase 2.}]
	\textbf{Step 1.} Introduce a new vertex $v_1$ and expose the edge $yv_1$, which is red.
	
	\textbf{Step 2.} Starting with $i=1$, Builder iteratively connects $v_i$ to two new vertices $u_i$ and $v_{i+1}$ by exposing the edges $v_iu_i$ and $v_iv_{i+1}$. If these two edges are one red and one blue, then without loss of generality assume that $v_iu_i$ is blue and $v_iv_{i+1}$ is red; Builder then continues the same procedure from $v_{i+1}$. This iteration stops as soon as there exists some $t\ge 1$ such that both $v_tu_t$ and $v_tv_{t+1}$ are blue, or when the total number of vertices reaches $\alpha$.
	
	If the total number of vertices reaches $\alpha$ while $v_tu_t$ and $v_tv_{t+1}$ are one red and one blue, then go to \textbf{End~2}. If the total number of vertices is at most $\alpha$ and both $v_tu_t$ and $v_tv_{t+1}$ are blue, we proceed as follows.
	
	For each $1\le j\le t-1$, Builder exposes the edge $v_ju_{j+1}$ (if it exists); these edges must be blue. Finally, expose the edge $yv_{t+1}$:
	\begin{itemize}
		\item If $yv_{t+1}$ is blue, then we call this execution of Phase~2 an instance of \emph{\textbf{Process~2}}.
		\item If $yv_{t+1}$ is red, then Builder additionally exposes the edge $yu_1$, which must be blue; in this case we call the execution an instance of \emph{\textbf{Process~3}}.
	\end{itemize}
	In either outcome, Builder has obtained a graph containing a longer blue path, and we proceed to \textbf{End~1}.
	
	\item[\textbf{End 1.}]
	If the total number of vertices in Builder's graph has reached $\alpha$, then terminate the algorithm. Otherwise, take the current graph (which contains a longer blue path) as the new input $G$ and iterate from \textbf{Start}. There is one exception: if immediately after completing Step~1 of Phase~2 the graph already has $\ell-3$ vertices, then we do not terminate; instead, we continue with Step~2 until the graph reaches $\ell-1$ vertices, and then terminate.
	
	\item[\textbf{End 2.}]
	In this situation, without loss of generality, $v_tu_t$ is blue and $v_tv_{t+1}$ is red, and the total number of vertices has reached $\alpha$. We distinguish the following cases:
	\begin{itemize}
		\item If $\alpha\in\{\ell-1,\ell-2\}$, we call this execution an instance of \emph{\textbf{Process~4}} and terminate the algorithm.
		\item 
		If $\alpha=\ell-3$, then Builder further connects $v_{t+1}$
		to two new vertices $u_{t+1}$ and $v_{t+2}$ by exposing the edges
		$v_{t+1}u_{t+1}$ and $v_{t+1}v_{t+2}$.
		If both of these edges are blue, then Builder exposes the edges
		$v_j u_{j+1}$ for $1\le j\le t$, all of which must be blue,
		and then exposes $yv_{t+2}$.
		If $yv_{t+2}$ is blue, we regard the completed execution as
		an instance of \textbf{Process~2}.
		If $yv_{t+2}$ is red, then Builder additionally exposes $yu_1$,
		which must be blue; we then regard the completed execution as
		an instance of \textbf{Process~3}.
		If instead the two edges $v_{t+1}u_{t+1}$ and $v_{t+1}v_{t+2}$
		are one red and one blue, then, after relabeling the two new
		vertices if necessary, we regard the completed execution as
		an instance of \textbf{Process~4}.
		In all cases, the graph has $\ell-1$ vertices, and the algorithm
		terminates.
	\end{itemize}
	
\end{description}

\noindent\textbf{Remark on the Blue Path Expansion Algorithm.}

The Blue Path Expansion Algorithm is a constructive strategy for Builder rather than an ``algorithm'' in the computational sense. Starting from a current longest blue path, Builder first attempts to extend it directly at an endpoint. If this attempt is blocked by a red edge, the strategy switches to a more delicate iterative procedure designed to overcome this obstacle. In each execution, the strategy either extends the blue path in a controlled way via \textbf{Process~1--3}, or it terminates in \textbf{Process~4}, producing a well-structured graph on $\alpha\in\{\ell-3,\ell-2,\ell-1\}$ vertices; these structural properties will be exploited later to construct a blue $C_\ell$.

Observe that the entire expansion strategy consists of a sequence of Processes~1--4. In each instance of \textbf{Process~1}, Builder spends one round and extends the blue path by one vertex. In each instance of \textbf{Process~2}, Builder spends $3t+1$ rounds and increases the blue path by $2t+1$ vertices. In each instance of \textbf{Process~3}, Builder spends $3t+2$ rounds and likewise increases the blue path by $2t+1$ vertices. If \textbf{Process~4} occurs, the strategy terminates; we will analyze this case in more detail later.

\begin{thm}\label{thm:G3}
	If Builder constructs $G_3$ within the first three rounds of the $(K_{1,3},C_\ell)$--online Ramsey game, then Builder can force a blue $C_\ell$ within a total of $N$ rounds.
\end{thm}

\begin{proof}
	We apply the Blue Path Expansion Algorithm to $G_3$ as the initial graph, thereby extending the blue $P_2$ contained in $G_3$; the vertex labeling of $G_3$ is as in~\cref{G3}.
	
	Suppose that during the entire expansion procedure, \textbf{Process~1} occurs $k_1$ times. Then Builder spends $k_1$ rounds and extends the blue path by $k_1$ vertices. If \textbf{Process~2} occurs $k_2$ times, then Builder spends
	$k_2+3\sum_{i=1}^{k_2} t_i$ rounds and extends the blue path by
	$k_2+2\sum_{i=1}^{k_2} t_i$ vertices. If \textbf{Process~3} occurs $k_3$ times, then Builder spends
	$2k_3+3\sum_{i=1}^{k_3} t'_i$ rounds and extends the blue path by
	$k_3+2\sum_{i=1}^{k_3} t'_i$ vertices. Finally, \textbf{Process~4} occurs $k_4\in\{0,1\}$ times.
	
	Note that in the first execution of the algorithm one must be in \textbf{Process~1}, \textbf{Process~2}, or \textbf{Process~4}. Indeed, the endpoint $y$ of the initial blue path in $G_3$ is incident with a red edge; once $yv_1$ is colored red, the vertex $y$ becomes a better vertex, so every subsequently exposed edge incident with $y$ must be blue. Similarly, after each occurrence of \textbf{Process~3}, the new endpoint $y$ of the resulting longest blue path is again incident with a red edge, and therefore the next execution must be \textbf{Process~1}, \textbf{Process~2}, or \textbf{Process~4}. Hence,
	\begin{equation}\label{equ1}
		k_3 \le k_1+k_2 .
	\end{equation}
	Moreover, if equality holds, then $y$ is incident with a red edge.
	Indeed, if $k_1=k_2=k_3=0$, then $y$ is an endpoint of the
	initial blue path in $G_3$. Otherwise, the last occurrence
	among \textbf{Processes 1--3} is \textbf{Process 3}, whose
	new endpoint $y$ is incident with a red edge.
	
	For convenience, let $k:=k_1+k_2+k_3$ and
	\[
	M:=\sum_{i=1}^{k_2} t_i+\sum_{i=1}^{k_3} t'_i .
	\]
	Let $G$ denote the graph produced when the algorithm terminates.
	
	\setcounter{case}{0}
	\begin{case}
		$k_4=0$.
	\end{case}
	\begin{figure}[htp]
		\centering
		\begin{tikzpicture}[global scale=1.2]
			\begin{scope}[line width=1pt, blue]
				\draw (1,1) to (2,1);
				\draw(2,1) to[bend left=30] (8,1);
			\end{scope}      	  		
			
			\begin{scope}[line width=1pt, red]
				\draw (1,1) to (1.5,2);
				\draw (2,1) to (1.5,2);
			\end{scope}
			
			\filldraw[fill=white] (1,1) circle (2pt);
			\filldraw[fill=white] (2,1) circle (2pt);
			\filldraw[fill=white] (1.5,2) circle (2pt);
			\filldraw[fill=white] (8,1) circle (2pt);
			
			\node at (1,0.7) {$x$};
			\node at (8,0.7) {$y$};
			\node at (1.5,2.3) {$v_0$};		
		\end{tikzpicture}
		\caption{}
		\label{k4=0}
	\end{figure}
	
	In this case, $G$ contains a blue path $P_{\alpha-1}$ obtained by extending the blue $P_2$ in $G_3$, and the unique vertex not on this path is $v_0$. We continue to denote the endpoints of this blue $P_{\alpha-1}$ by $x$ and $y$ (see~\cref{k4=0}).
	
	\setcounter{subcase}{0}
	\begin{subcase}
		$\alpha=\ell-1$, or $\alpha=\ell-2$, or $\alpha=\ell-3$ and $k_3\le k_1+k_2-1$.
	\end{subcase}
	Builder exposes the edge $yv_0$, which must be blue. This yields a better blue $P_\alpha$ with endpoints $x$ and $v_0$, constructed within $4+3M+k+k_3$ rounds (where the term $4$ accounts for the initial graph $G_3$ together with the edge $v_0y$). Observe that $\alpha = 3+2M+k$. Consequently,
	\begin{equation}\label{equ3}
		4+3M+k+k_3=\left\lfloor \frac{3\alpha-1+2k_3-k}{2}\right\rfloor .
	\end{equation}
	If $\alpha=\ell-1$ or $\alpha=\ell-2$, then by~\cref{equ1,equ3} Builder constructs a better $P_\alpha$ within at most
	$\left\lfloor\frac{3\alpha-1}{2}\right\rfloor$ rounds. For $\alpha=\ell-1$ we then apply~\cref{lem:main2}, and for $\alpha=\ell-2$ we apply~\cref{lem:betterpl-2}, completing the proof.
	If $\alpha=\ell-3$ and $k_3\le k_1+k_2-1$, then~\cref{equ3}
	gives that Builder constructs a better $P_\alpha$ within at most
	$\left\lfloor\frac{3\alpha-2}{2}\right\rfloor$ rounds.
	Since $\ell-\alpha=3$ is odd, the conclusion follows from
	the strengthened assertion in~\cref{lem:main1}.
	
	\begin{subcase}
		$\alpha=\ell-3$ and $k_3=k_1+k_2$.
	\end{subcase}
	In this case the endpoint $y$ is incident with a red edge. Builder introduces a new vertex $w$ and exposes the edge $yw$.
	\begin{itemize}
		\item If $yw$ is blue, then Builder exposes $wv_0$, which must be blue. Thus, within $5+3M+k+k_3$ rounds Builder constructs a better $P_{\alpha+1}$. By~\cref{equ3}, this takes at most $\left\lfloor\frac{3\alpha+1}{2}\right\rfloor$ rounds, and applying~\cref{lem:betterpl-2} finishes the proof.
		
		\item If $yw$ is red, then $y$ becomes a better vertex. Builder next exposes the edge $xw$.
		\begin{itemize}
			\item If $xw$ is blue, then Builder introduces a new vertex $z$ and exposes the edges $zv_0$ and $yz$, both of which must be blue. Hence Builder constructs a better $P_{\alpha+2}$ within $7+3M+k+k_3$ rounds, that is, within at most
			$\left\lfloor\frac{3\alpha+5}{2}\right\rfloor$ rounds, and we then apply~\cref{lem:main2}.
			
			\item If $xw$ is also red, then Builder introduces new vertices $u$ and $v$ and sequentially exposes the edges $xu$, $uv_0$, $yv$, $vw$, and $wv_0$, all of which must be blue. Consequently, within $10+3M+k+k_3$ rounds Builder constructs a blue $C_{\alpha+3}$, i.e., within at most
			$\left\lfloor\frac{3\alpha+11}{2}\right\rfloor$ rounds Builder completes the construction of $C_\ell$.
		\end{itemize}
	\end{itemize}
	
	\begin{case}
		$k_4=1$.
	\end{case}
	\begin{figure}[htp]
		\centering
		\begin{tikzpicture}[global scale=1.2]
			\begin{scope}[line width=1pt, blue]
				\draw (0,1) -- (1,1);
				\draw (4,1) -- (4,2);
				\draw (4,1) -- (5,2);
				\draw (5,1) -- (5,2);
				\draw (6,1) -- (6,2);
				\draw (6,1) -- (7,2);
				\draw (7,1) -- (7,2);
				\draw (7,1) -- (8,2);
				\draw (1,1) to[bend left=30] (3,1);
			\end{scope}
			
			\begin{scope}[line width=1pt, red]
				\draw (0,1) -- (0.5,2);
				\draw (1,1) -- (0.5,2);
				\draw (3,1) -- (4,1);
				\draw (4,1) -- (5,1);
				\draw (6,1) -- (7,1);
				\draw (7,1) -- (8,1);
			\end{scope}
			
			\begin{scope}[line width=1pt, red, dotted]
				\draw (4.5,1) -- (5,1);
				\draw (5,1) -- (6,1);
			\end{scope}
			
			\filldraw[fill=white] (0.5,2) circle (2pt); % v_0
			\filldraw[fill=white] (0,1) circle (2pt);   % x
			
			\filldraw[
			draw=black,
			fill=white,
			line width=0.5pt,
			dash pattern=on 1pt off 0.5pt
			] (1,1) circle (2.3pt);
			
			\filldraw[fill=white] (3,1) circle (2pt);   % y
			\filldraw[fill=white] (4,1) circle (2pt);   % v_1
			\filldraw[fill=white] (5,1) circle (2pt);
			\filldraw[fill=white] (6,1) circle (2pt);
			\filldraw[fill=white] (7,1) circle (2pt);   % v_t
			\filldraw[fill=white] (8,1) circle (2pt);   % v_{t+1}
			
			\filldraw[fill=white] (4,2) circle (2pt);   % u_1
			\filldraw[fill=white] (5,2) circle (2pt);
			\filldraw[fill=white] (6,2) circle (2pt);
			\filldraw[fill=white] (7,2) circle (2pt);   % u_t
			\filldraw[fill=white] (8,2) circle (2pt);   % u_{t+1}
			
			\node at (0,0.7) {$x$};
			\node at (0.5,2.3) {$v_0$};
			\node at (3,0.7) {$y$};
			\node at (4,0.7) {$v_1$};
			\node at (7,0.7) {$v_t$};
			\node at (4,2.3) {$u_1$};
			\node at (7,2.3) {$u_t$};
			\node at (8,2.3) {$u_{t+1}$};
			\node at (8,0.7) {$v_{t+1}$};
		\end{tikzpicture}
		\caption{The vertex indicated by the dashed circle may coincide with $y$. In that case, the blue subpath between them is absent.}
		\label{k4=1}
	\end{figure}
	
	In this case, after the algorithm terminates, Builder introduces a new vertex $u_{t+1}$ and, for each $1\le j\le t$, exposes the edge $v_ju_{j+1}$; all these edges must be blue. The resulting graph is shown schematically in~\cref{k4=1}. At this point Builder has spent $3M+k+k_3+3t+4$ rounds in total, and the graph has exactly $\alpha+1$ vertices. Note that $\alpha = 2M+k+2t+4$. Therefore,
	\begin{equation}\label{equ5}
		3M+k+k_3+3t+4 \le \left\lfloor \frac{3\alpha+k_3-k_2-k_1}{2}\right\rfloor -2 .
	\end{equation}
	
	\setcounter{subcase}{0}
	\begin{subcase}
		$\alpha=\ell-1$, or $\alpha=\ell-2$ and $k_3\le k_1+k_2-1$.
	\end{subcase}
	Builder exposes the edge $yv_{t+1}$.
	\begin{itemize}
		\item If $yv_{t+1}$ is red, then Builder further exposes
		$xv_{t+1}$, $yu_1$, and $u_{t+1}v_0$, all of which
		must be blue. In this way, Builder spends $4$ additional
		rounds and obtains a better $P_{\alpha+1}$ with endpoints
		$v_{t+1}$ and $v_0$.
		
		\item If $yv_{t+1}$ is blue, then Builder exposes $xu_1$.
		\begin{itemize}
			\item If $xu_1$ is also blue, then Builder exposes
			$v_0u_{t+1}$, which must be blue. In this case,
			Builder spends $3$ additional rounds and obtains
			a better $P_{\alpha+1}$ with endpoints $v_{t+1}$
			and $v_0$.
			
			\item If $xu_1$ is red, then Builder exposes
			$v_0v_{t+1}$ and $xu_{t+1}$, both of which must
			be blue. In this case, Builder spends $4$ additional
			rounds and obtains a better $P_{\alpha+1}$ with
			endpoints $u_1$ and $v_0$.
		\end{itemize}
	\end{itemize}
	
	Summarizing, when $\alpha=\ell-1$, by~\cref{equ1,equ5} Builder constructs a better $P_{\alpha+1}$ whose endpoints are non-adjacent within at most $\left\lfloor\frac{3\alpha}{2}\right\rfloor+2$ rounds. Builder then joins the two endpoints; this edge must be blue, and hence Builder constructs $C_\ell$ within $\left\lfloor\frac{3\alpha}{2}\right\rfloor+3$ rounds.
	
	When $\alpha=\ell-2$ and $k_3\le k_1+k_2-1$,~\cref{equ5} implies that Builder constructs a better $P_{\alpha+1}$ within at most
	$\left\lfloor\frac{3\alpha-1}{2}\right\rfloor+2$ rounds, and the conclusion follows from~\cref{lem:main2}.
	
	\begin{subcase}
		$\alpha=\ell-2$ and $k_3=k_1+k_2$.
	\end{subcase}
	In this case $y$ becomes a better vertex. Builder exposes $v_{t+1}u_{t+1}$. If it is blue, then Builder exposes $yv_{t+1}$ and $v_0u_1$, both of which must be blue. Together with~\cref{equ5}, this shows that Builder constructs a better $P_{\alpha+1}$ within at most $\left\lfloor\frac{3\alpha}{2}\right\rfloor+1$ rounds, and we then apply~\cref{lem:main2}.
	If instead $v_{t+1}u_{t+1}$ is red, then Builder exposes $xv_{t+1}$, $v_{t+1}u_1$, and $u_{t+1}v_0$, all of which must be blue. Hence Builder constructs a best $P_{\alpha+1}$ within at most $\left\lfloor\frac{3\alpha}{2}\right\rfloor+2$ rounds, and applying~\cref{lem:main2} completes the proof.
\end{proof}

\begin{proof}[of the upper bound in Theorem~\ref{thm:mainresult}]
	We now complete the proof of the upper bound in \cref{thm:mainresult}. First, by \cref{obs:first}, Builder forces one of the configurations $G_1$, $G_2$, $G_3$, or $G_4$. If Builder obtains $G_1$, $G_2$, or $G_3$, then the desired conclusion follows from \cref{thm:G1}, \cref{thm:G2}, and \cref{thm:G3}, respectively. If instead Builder obtains $G_4$ (see \cref{G4}), then Builder introduces two new vertices $u$ and $w$, and exposes the edges $v_1u$, $uv_3$, $v_3w$, and $wv_2$, all of which must be blue. Hence Builder constructs a best $P_5$ within $7$ rounds, and applying \cref{lem:main1} completes the proof.
	This finishes the proof of the upper bound.
\end{proof}

\section{Proof of the Main Lemmas}\label{sec4}

We will use as a lemma the exact value of $\tilde r(K_{1,3},P_\ell)$ established by Song, Wang, and Zhang~\cite{Song2025}.

\begin{lemma}\label{thm:Song2025}
	For $\ell\ge 2$, we have $\tilde r(K_{1,3},P_\ell)=\left\lfloor \frac{3\ell}{2}\right\rfloor$.
\end{lemma}

Next we prove all lemmas stated in Section~\ref{sec3}. For each lemma, we first recall its statement and then provide the proof.

\begin{customlemma}{\ref{lem:main1}}
	Let $m\ge 2$ and $\ell-m\ge 2$. Suppose that Builder has constructed
	a blue $P_m$ within $k$ rounds and that one of the following holds:
	\begin{itemize}
		\item $P_m$ is a better path and $2k\le 3m-3$;
		\item $P_m$ is a best path and $2k\le 3m-1$.
	\end{itemize}
	Then Builder can force a blue $C_\ell$ within a total of
	$N$ rounds.
	Moreover, if $\ell-m$ is odd, then the condition $2k\le 3m-3$
	in the first case can be relaxed to $2k\le 3m-2$.
\end{customlemma}
\begin{proof}
	By \cref{thm:Song2025}, Builder can force a blue path $P_{\ell-m}$ that is vertex-disjoint from the existing blue path $P_m$ within $\left\lfloor \frac{3(\ell-m)}{2}\right\rfloor$ additional rounds. Let $u$ and $v$ be the endpoints of $P_m$, and let $w$ and $x$ be the endpoints of $P_{\ell-m}$.
	
	First assume that $P_m$ is a better path. Let $u$ be the good endpoint and $v$ the better endpoint. Builder exposes the edges $uw$ and $ux$, at least one of which must be blue; without loss of generality, suppose that $uw$ is blue. Builder then exposes the edge $vx$, which must be blue. Consequently, Builder creates a blue $C_\ell$ within $\left\lfloor \frac{3(\ell-m)}{2}\right\rfloor + k + 3 \le N$ rounds in total. Moreover, if $\ell-m$ is odd and $2k\le 3m-2$, then
	$\left\lfloor\frac{3(\ell-m)}{2}\right\rfloor+k+3
	\le \left\lfloor\frac{3m-2}{2}\right\rfloor
	+\frac{3(\ell-m)-1}{2}+3=N$.
	
	Now assume that $P_m$ is a best path, so both $u$ and $v$ are better vertices. Builder exposes the edges $uw$ and $vx$, both of which must be blue. Hence Builder creates a blue $C_\ell$ within $\left\lfloor \frac{3(\ell-m)}{2}\right\rfloor + k + 2 \le N$ rounds in total.
\end{proof}

\begin{customlemma}{\ref{lem:main2}}
	Suppose Builder has constructed a blue $P_{\ell-1}$ within $k$ rounds.
	\begin{itemize}
		\item If $P_{\ell-1}$ is a better path, then Builder can force a blue $C_\ell$ within a total of $k+3$ rounds.
		\item If $P_{\ell-1}$ is a best path, then Builder can force a blue $C_\ell$ within a total of $k+2$ rounds.
	\end{itemize}
\end{customlemma}
\begin{proof}
	Let $u$ and $v$ be the endpoints of the blue $P_{\ell-1}$ constructed within $k$ rounds.
	
	First assume that $P_{\ell-1}$ is a better path. Let $u$ be the good endpoint and $v$ the better endpoint. Builder introduces two new vertices $w$ and $z$, and exposes the edges $uw$ and $uz$. Since $u$ is good, at least one of these two edges must be blue; without loss of generality, assume that $uw$ is blue. Builder then exposes the edge $vw$, which must be blue because $v$ is a better vertex. This yields a blue $C_\ell$ within $k+3$ rounds in total.
	
	Now assume that $P_{\ell-1}$ is a best path, so both $u$ and $v$ are better vertices. Builder introduces a new vertex $w$ and exposes the edges $uw$ and $vw$, both of which must be colored blue. Hence Builder creates a blue $C_\ell$ within $k+2$ rounds in total.
\end{proof}

\begin{customlemma}{\ref{lem:betterpl-2}}
	For $\ell\ge 4$, if Builder has already constructed a better $P_{\ell-2}$, then Builder can create a blue $C_{\ell}$ within five additional rounds.
\end{customlemma}
\begin{proof}
	Let the endpoints of the better $P_{\ell-2}$ be $u$ and $v$, where $u$ is the better endpoint. Builder introduces a new vertex $w$ and exposes the edge $vw$.
	
	If $vw$ is blue, then Builder introduces three new vertices $x,y,z$ and exposes the edges $wx$, $wy$, and $wz$, among which at least one must be blue. Without loss of generality, assume that $wx$ is blue. Builder then exposes the edge $ux$, which must be blue since $u$ is a better vertex. Thus the cycle $xuP_{\ell-2}vwx$ is a blue $C_{\ell}$.
	
	If $vw$ is red, then Builder introduces a new vertex $x$
	and exposes the edge $wx$. If $wx$ is blue, then Builder
	exposes the edges $vx$ and $wu$, both of which must be blue,
	and hence $xwuP_{\ell-2}vx$ forms a blue $C_\ell$.
	If instead $wx$ is red, then Builder introduces a new vertex
	$y$ and exposes the edges $wy$, $yv$, and $wu$, all of which
	must be blue, yielding the blue cycle $ywuP_{\ell-2}vy$.
\end{proof}

\begin{customlemma}{\ref{lem:type}}
	Assume that $k$ and $m$ satisfy $2k\le 3m-5$ and $\ell-m\ge 6$.
	If Builder constructs a Type~1 configuration within $k+4$ rounds, or a Type~2 configuration within $k+2$ rounds, or a Type~3 configuration within $k+3$ rounds, then in each case Builder can force a blue $C_\ell$ within a total of $N$ rounds.
\end{customlemma}
\begin{proof}
	Suppose first that Builder has obtained a Type~1 configuration. Builder exposes the edge $zv_1$. If $zv_1$ is blue, then within $k+5$ rounds Builder has constructed a better $P_{m+3}$. If $zv_1$ is red, then Builder introduces a new vertex $u$ and exposes the edges $yu$ and $zu$, both of which must be blue. In this case, within $k+7$ rounds Builder has constructed a better $P_{m+4}$. In either outcome, \cref{lem:main1} applies and yields the desired conclusion.
	
	Next assume that Builder has obtained a Type~2 configuration. Builder first exposes the edge $xv_1$, and we distinguish two cases.
	\begin{enumerate}
		\item If $xv_1$ is blue, then Builder next exposes the edge $xy$.
		\begin{itemize}
			\item If $xy$ is blue, then within $k+4$ rounds Builder obtains a better $P_{m+2}$.
			\item If $xy$ is red, then Builder introduces a new vertex $z$ and exposes the edges $yz$ and $zx$, both of which must be blue. Thus, within $k+6$ rounds Builder obtains a best $P_{m+3}$.
		\end{itemize}
		\item If $xv_1$ is red, then Builder next exposes the edge $yv_1$.
		\begin{itemize}
			\item If $yv_1$ is blue, then Builder additionally exposes the edge $xy$, which must be blue, and hence within $k+5$ rounds Builder constructs a best $P_{m+2}$.
			\item If $yv_1$ is red, then Builder introduces new vertices $w$ and $u$, and exposes the edges $v_mw$, $wx$, $xu$, and $uy$, all of which must be blue. Consequently, within $k+8$ rounds Builder constructs a best $P_{m+4}$.
		\end{itemize}
	\end{enumerate}
	In all subcases above, the resulting better/best path satisfies the assumptions of \cref{lem:main1}, which completes the proof for Type~2.
	
	Finally, suppose that Builder has obtained a Type~3 configuration. Builder exposes the edge $yv_1$, which must be blue. Hence within $k+4$ rounds Builder constructs a better $P_{m+2}$, and the conclusion follows from \cref{lem:main1}.
\end{proof}

Next we prove the more involved \cref{lem:G1}. We first establish two auxiliary lemmas that will serve as the main tools. Starting from an already constructed wavy path, we will use an iterative procedure to obtain a longer wavy path in relatively few additional rounds, thereby saving enough rounds for the final construction of a blue cycle.

\begin{lemma}\label{lem:wavy}
	Assume that $i\ge 4$. If a wavy path $P(i,j)$ satisfies that both $v_1v_2$ and $v_{i-1}v_i$ are blue, then $i\ge 2j+2$.
\end{lemma}

\begin{proof}
	Each wavy triangle contains exactly one wavy point and two non-wavy vertices. Hence $P(i,j)$ has exactly $j$ wavy points, and exactly $2j$ non-wavy vertices that lie in wavy triangles.
	
	Note that $v_1$ and $v_i$ are the two endpoints of the underlying blue path $P_{i+j}$ in $P(i,j)$, and by assumption the edges $v_1v_2$ and $v_{i-1}v_i$ are blue. Therefore, neither $v_1$ nor $v_i$ can serve as a non-wavy vertex of any wavy triangle in $P(i,j)$. Consequently, the total number of vertices $i+j$ satisfies $i+j \ge (2j)+j+2$, and thus $i\ge 2j+2$.
\end{proof}

\begin{lemma}\label{lem:wavyG1}
	Let $i+j\le \ell-3$ and $i\ge 4$. Suppose Builder has constructed a copy of $P(i,j)$ within $i+2j-1$ rounds and that the edge $v_1v_2$ is blue. Then Builder can either construct a copy of $P(i+1,j)$ within a total of $i+2j$ rounds, or construct a copy of $P(i+1,j+1)$ within a total of $i+2j+2$ rounds, or else force a blue $C_\ell$ within a total of $N$ rounds.
\end{lemma}

\begin{proof}
	Builder has constructed $P(i,j)$ within $i+2j-1$ rounds. Builder now introduces a new vertex $v_{i+1}$ and exposes the edge $v_iv_{i+1}$.
	
	If $v_iv_{i+1}$ is blue, then $P(i,j)\cup\{v_iv_{i+1}\}$ forms a copy of $P(i+1,j)$. Since $i\ge 2j+1$, we have $i+1\ge 2j+1$, and thus Builder completes $P(i+1,j)$ within $i+2j$ rounds.
	
	Assume next that $v_iv_{i+1}$ is red and $v_{i-1}v_i$ is also red. Then $v_{i-1}u_jv_iv_{i-1}$ forms a wavy triangle in $P(i,j)$. Hence, within $i+2j$ rounds Builder has obtained a Type~1 configuration. Moreover, $P(i,j)\setminus\{u_j,v_i,v_{i+1}\}$ is a copy of $P(i-1,j-1)$, and from $i\ge 2j+1$ it follows that $i-1\ge 2(j-1)+1$. In particular, this Type~1 configuration has at most $i+j+1\le \ell-2$ vertices, and its blue main path corresponds to the blue path $P_{i+j-2}$ inside $P(i-1,j-1)$.
	
	Finally, assume that $v_iv_{i+1}$ is red while $v_{i-1}v_i$ is blue. Builder then introduces a new vertex $u_{j+1}$ and exposes the edge $v_iu_{j+1}$. If this edge is red, then within $i+2j+1$ rounds Builder has obtained a Type~2 configuration whose blue main path is the blue $P_{i+j}$ inside $P(i,j)$. If instead $v_iu_{j+1}$ is blue, then Builder exposes the edge $v_{i+1}u_{j+1}$. If $v_{i+1}u_{j+1}$ is red, then within $i+2j+2$ rounds Builder has obtained a Type~3 configuration whose blue main path is again the blue $P_{i+j}$ inside $P(i,j)$. If $v_{i+1}u_{j+1}$ is blue, then Builder has constructed $P(i+1,j+1)$. Indeed, since both $v_{i-1}v_i$ and $v_1v_2$ are blue, \cref{lem:wavy} implies $i\ge 2j+2$, so $i+1\ge 2(j+1)+1$, and thus Builder completes $P(i+1,j+1)$ within $i+2j+2$ rounds.
	
	It remains to show that in each of the above outcomes where a Type~$\beta$ configuration is formed (with $\beta\in\{1,2,3\}$), Builder can force a blue $C_\ell$ within a total of $N$ rounds. Throughout, the labeling of $P(i,j)$ follows our standardized convention. Let $\alpha$ denote the number of vertices of the obtained Type~$\beta$. Since Type~$\beta$ is obtained from $P(i,j)$ by adding one or two new vertices and $i\ge 4$, we have $6\le \alpha \le \ell-1$. Moreover, in the construction of Type~1 we necessarily have $j\ge 1$.
	
	\setcounter{subcase}{0}
	\setcounter{case}{0}
	\begin{case}
		$\alpha\in\{\ell-1,\ell-2,\ell-3\}$.
	\end{case}
	
	\begin{subcase}
		Type~1.
	\end{subcase}
	In Type~1, Builder adds the vertex $v_{i+1}$ to $P(i,j)$, and hence $\alpha\neq \ell-1$. Since $i+j+1=\alpha$ and $i\ge 2j+1$, we have $j\le \left\lfloor\frac{\alpha-2}{3}\right\rfloor$. Builder exposes the edge $v_1v_{i+1}$. If $v_1v_{i+1}$ is blue, then Builder constructs a better $P_\alpha$ within $i+2j+1=\alpha+j\le \left\lfloor\frac{4\alpha-2}{3}\right\rfloor$ rounds. For $\alpha=\ell-2$ or $\ell-3$, we may apply~\cref{lem:main1} to finish. If instead $v_1v_{i+1}$ is red, then Builder introduces a new vertex $u$ and exposes the edges $v_iu$ and $uv_{i+1}$, both of which must be blue. Thus Builder constructs a better $P_{\alpha+1}$ within $i+2j+3 \le \left\lfloor\frac{4\alpha+4}{3}\right\rfloor$ rounds. When $\alpha=\ell-3$ we apply~\cref{lem:main1}, and when $\alpha=\ell-2$ we apply~\cref{lem:main2}, completing the proof in both cases.
	
	\begin{subcase}
		Type~2.
	\end{subcase}
	In Type~2, Builder adds the vertices $v_{i+1}$ and $u_{j+1}$ to $P(i,j)$, so $i+j+2=\alpha$. From $i\ge 2j+1$ we obtain $j\le \left\lfloor\frac{\alpha-3}{3}\right\rfloor$. Builder exposes the edge $v_1u_{j+1}$.
	\begin{itemize}
		\item Suppose that $v_1u_{j+1}$ is red.
		\begin{itemize}
			\item If $\alpha=\ell-1$, then Builder introduces a new vertex $w$ and exposes the edges $v_iw$ and $wu_{j+1}$, both of which must be blue. Hence Builder constructs a better $P_\alpha$, namely $v_1P_{i+j}v_iwu_{j+1}$, within $i+2j+4=\alpha+j+2\le \left\lfloor\frac{4\alpha+3}{3}\right\rfloor$ rounds, and applying~\cref{lem:main2} yields the conclusion.
			
			\item If $\alpha=\ell-2$, Builder exposes the edge $v_1v_{i+1}$. If $v_1v_{i+1}$ is blue, then Builder exposes $u_{j+1}v_{i+1}$, which must be blue, and thus constructs a best $P_\alpha$ within
			$i+2j+4\le \left\lfloor\frac{4\alpha+3}{3}\right\rfloor$ rounds; \cref{lem:main1} then applies. If $v_1v_{i+1}$ is red, then Builder introduces new vertices $u$ and $v$ and exposes the edges
			$v_1u$, $uu_{j+1}$, $u_{j+1}v_{i+1}$, $v_{i+1}v$, and $vv_i$, all of which are blue. Consequently, Builder constructs a blue $C_\ell$ within $i+2j+8 \le \left\lfloor\frac{4\ell+7}{3}\right\rfloor \le N$ rounds.
			
			\item If $\alpha=\ell-3$, Builder again exposes $v_1v_{i+1}$. If $v_1v_{i+1}$ is blue, then exposing $u_{j+1}v_{i+1}$ (necessarily blue) yields a best $P_\alpha$ within
			$i+2j+4\le \left\lfloor\frac{4\alpha+3}{3}\right\rfloor$ rounds, and \cref{lem:main1} applies. If $v_1v_{i+1}$ is red, then Builder introduces new vertices $u,v,w$ and exposes the edges
			$v_1u$, $uu_{j+1}$, $u_{j+1}w$, $wv_{i+1}$, $v_{i+1}v$, and $vv_i$, all of which are blue. Thus Builder constructs a blue $C_\ell$ within $i+2j+9 \le \left\lfloor\frac{4\ell+6}{3}\right\rfloor$ rounds.
		\end{itemize}
		
		\item Suppose that $v_1u_{j+1}$ is blue. Builder then exposes the edge $u_{j+1}v_{i+1}$.
		\begin{itemize}
			\item If $u_{j+1}v_{i+1}$ is blue, then within $i+2j+3=\alpha+j+1\le \left\lfloor\frac{4\alpha}{3}\right\rfloor$ rounds Builder constructs a better $P_\alpha$. For $\alpha=\ell-2$ or $\ell-3$ we apply~\cref{lem:main1}, and for $\alpha=\ell-1$ we apply~\cref{lem:main2}.
			
			\item If $u_{j+1}v_{i+1}$ is red, then Builder introduces a new vertex $x$ and exposes the edges $u_{j+1}x$ and $v_{i+1}x$, both of which must be blue. Hence Builder constructs a best $P_{\alpha+1}$ within $i+2j+5=\alpha+j+3\le \left\lfloor\frac{4\alpha+6}{3}\right\rfloor$ rounds. When $\alpha=\ell-3$ we apply~\cref{lem:main1}, and when $\alpha=\ell-2$ we apply~\cref{lem:main2}. When $\alpha=\ell-1$, Builder exposes the edges $v_1v_i$ and $v_2v_{i+1}$, both of which are blue, and thus constructs a blue $C_\ell$ within $\left\lfloor\frac{4\ell+8}{3}\right\rfloor \le N$ rounds, namely $v_1v_iP_{i+j-1}v_2v_{i+1}xu_{j+1}v_1$.
		\end{itemize}
	\end{itemize}
	
	\begin{subcase}
		Type~3.
	\end{subcase}
	In Type~3, Builder adds the vertices $v_{i+1}$ and $u_{j+1}$ to $P(i,j)$, so $i+j+2=\alpha$. Since $i\ge 2j+1$, we have $j\le \left\lfloor\frac{\alpha-3}{3}\right\rfloor$. Builder exposes the edge $v_1v_{i+1}$, which must be blue. Therefore Builder constructs a better $P_\alpha$ within
	$i+2j+3=\alpha+j+1\le \left\lfloor\frac{4\alpha}{3}\right\rfloor$
	rounds. For $\alpha=\ell-2$ or $\ell-3$ we apply~\cref{lem:main1}, and for $\alpha=\ell-1$ we apply~\cref{lem:main2}, completing the proof.
	
	\begin{case}
		$\alpha\le \ell-4$.
	\end{case}
	For Type~2 and Type~3, the blue main path is the blue $P_{i+j}$ contained in $P(i,j)$, which was built within $i+2j-1$ rounds. Here $\alpha=i+j+2\le \ell-4$, and thus $\ell-(i+j)\ge 6$. A direct calculation gives
	\[
	2(i+2j-1)-3(i+j)+5 = j-i+3 \le 2-j .
	\]
	Since $i\ge 4$, it follows in all cases (whether $j=0$, $j=1$, or $j\ge 2$) that
	\[
	2(i+2j-1)\le 3(i+j)-5.
	\]
	
	For Type~1, the blue main path is the blue $P_{i+j-2}$ contained in $P(i-1,j-1)$. Here $\alpha=i+j+1\le \ell-4$, and hence $\ell-(i-1+j-1)\ge 6$. Similarly,
	\[
	2\bigl((i-1)+2(j-1)-1\bigr)\le 3\bigl((i-1)+(j-1)\bigr)-5.
	\]
	
	Therefore, in all cases above the hypotheses of~\cref{lem:type} are satisfied, and Builder can force a blue $C_\ell$ within $N$ rounds.
\end{proof}

\begin{customlemma}{\ref{lem:G1}}
	Let $i+j\le \ell-3$ and $i\ge 4$. If Builder has constructed a copy of $P(i,j)$ within $i+2j-1$ rounds and the edge $v_1v_2$ is blue, then Builder can force a blue $C_\ell$ within a total of $N$ rounds.
\end{customlemma}

\begin{proof}
	Starting from $P(i,j)$, we repeatedly apply~\cref{lem:wavyG1}. During this procedure, either Builder already forces a blue $C_\ell$ within $N$ rounds (and we are done), or else Builder keeps obtaining longer wavy paths, eventually reaching a wavy path $P(m,n)$ with $m+n=\ell-2$ or $m+n=\ell-1$. Note that in each application of~\cref{lem:wavyG1}, whenever a new wavy path $P(i,j)$ is obtained, the total number of rounds used so far is exactly the number of edges of $P(i,j)$. We now consider the two possible terminal cases.
	
	\setcounter{case}{0}
	\begin{case}
		$m+n=\ell-2$.
	\end{case}
	In this case, $n\le \left\lfloor\frac{\ell-3}{3}\right\rfloor$. Builder introduces a new vertex $v_{m+1}$ and exposes the edge $v_mv_{m+1}$. If $v_mv_{m+1}$ is blue, then we immediately obtain a wavy path on $\ell-1$ vertices, which we treat in the next case. So assume that $v_mv_{m+1}$ is red. Builder then exposes the edge $v_1v_{m+1}$.
	\begin{itemize}
		\item If $v_1v_{m+1}$ is red, then Builder exposes $v_2v_{m+1}$, which must be blue. Hence Builder constructs a better $P_{\ell-2}$, namely $v_mP_{m+n-1}v_2v_{m+1}$, within $m+2n+2 \le \left\lfloor\frac{4\ell-3}{3}\right\rfloor$ rounds. Applying~\cref{lem:betterpl-2} completes the proof.
		\item If $v_1v_{m+1}$ is blue, then Builder introduces new vertices $u,v,w$ and exposes the edges $v_{m+1}u$, $v_{m+1}v$, and $v_{m+1}w$, among which at least two are blue. Without loss of generality, assume that $v_{m+1}u$ and $v_{m+1}v$ are blue. Builder then exposes $v_mu$ and $v_mv$, at least one of which must be blue; without loss of generality, assume that $v_mu$ is blue. Then $uv_{m+1}v_1P_{m+n}v_mu$ is a blue $C_\ell$. The total number of rounds is at most $\left\lfloor\frac{4\ell+9}{3}\right\rfloor \le N$, as required.
	\end{itemize}
	
	\begin{case}
		$m+n=\ell-1$.
	\end{case}
	In this case, $n\le \left\lfloor\frac{\ell-2}{3}\right\rfloor$.
	Builder introduces four new vertices $x,w,y,z$ and exposes
	the edges $v_1x$, $v_1w$, $v_1y$, and $v_1z$.
	We distinguish the following cases.
	\begin{itemize}
		\item Suppose that exactly two of these four edges are blue.
		Without loss of generality, assume that $v_1x$ and $v_1w$
		are blue. Then $v_1$ becomes a better vertex. Next Builder
		introduces a new vertex $u$ and exposes the edge $v_mu$.
		\begin{itemize}
			\item If $v_mu$ is blue, then Builder exposes $v_1u$,
			which must be blue, and hence constructs a blue $C_\ell$
			within $m+2n+5$ rounds.
			
			\item If $v_mu$ is red, then Builder exposes $v_mx$
			and, if it is red, $v_mw$, which must then be blue.
			This yields a blue $C_\ell$ within $m+2n+6$ rounds.
			If $v_m$ is good in $P(m,n)$, then $v_m$ becomes
			better when $v_mu$ is colored red, so $v_mx$ must
			be blue and at most $m+2n+5$ rounds are needed.
		\end{itemize}
		
		\item Suppose that at least three of the edges $v_1x$,
		$v_1w$, $v_1y$, $v_1z$ are blue. Without loss of generality,
		assume that $v_1x$, $v_1w$, and $v_1y$ are blue.
		Builder then exposes the edges $v_mx$, $v_mw$, and $v_my$
		in this order, stopping as soon as a blue edge appears.
		At least one of these edges must be blue, yielding a blue
		$C_\ell$ within $m+2n+6$ rounds. If $v_m$ is good in
		$P(m,n)$, then at least one of $v_mx$ and $v_mw$ must
		be blue, so at most $m+2n+5$ rounds are needed.
	\end{itemize}
	If $v_m$ is good in $P(m,n)$, the total number of rounds is
	at most $m+2n+5\le\lfloor(4\ell+10)/3\rfloor$.
	Otherwise, $v_m$ belongs to no wavy triangle, so by the definition
	of a wavy path and our labeling convention, $v_{m-1}v_m$ is blue. Since $v_1v_2$ is also blue,
	\Cref{lem:wavy} gives $m\ge2n+2$, and hence
	$n\le\lfloor(\ell-3)/3\rfloor$. In this case, the total number
	of rounds is at most
	$m+2n+6\le\lfloor(4\ell+12)/3\rfloor$.
	Thus, in all cases above, Builder constructs a blue $C_\ell$
	within at most
	\[
	\left\lfloor\frac{4\ell+12}{3}\right\rfloor
	\le \left\lfloor\frac{3(\ell+1)}{2}\right\rfloor=N
	\]
	rounds, where the inequality holds for all $\ell\ge13$.
	This completes the proof.
\end{proof}	

\section*{Acknowledgements}
Y. Zhang was partially supported by the National Natural Science Foundation of China (NSFC) under Grant No. 12671403 and by the Natural Science Foundation of Hebei Province under Grant No. A2026205045.
H. Zhi was supported by the Graduate Innovation Fund of
Hebei Normal University under Grant No. ycxzzbs202605.

\section*{Data Availability Statement}

No data was used or generated in this research.

\section*{Conflict of Interest}

The authors declare that they have no known competing financial interests or personal relationships that could have appeared to influence the work reported in this paper.


\begin{thebibliography}{20}
	\bibitem{Adamska2025}
	N. Adamska and G. Adamski,
	Lower bounds for online size Ramsey numbers for paths,
	{\it arXiv:2504.14926} (2025).% \url{https://doi.org/10.48550/arXiv.2504.14926}.
	
	\bibitem{Adamski2024a}
	G. Adamski and M. Bednarska-Bzd\k{e}ga,
	Online Size Ramsey Numbers: Odd Cycles vs Connected Graphs,
	{\it Electron. J. Combin.} {\bf 31} (2024), \#P3.16.
	
	\bibitem{Adamski2024b}
	G. Adamski and M. Bednarska-Bzd\k{e}ga,
	Online size Ramsey numbers: Path vs $C_4$,
	{\it Discrete Math.} {\bf 347} (2024), 114214.
	
	\bibitem{Adamski2024c}
	G. Adamski, M. Bednarska-Bzd\k{e}ga, and V. Bla\v{z}ej,
	Online Ramsey numbers: Long versus short cycles,
	{\it SIAM J. Discrete Math.} {\bf 38(4)} (2024), 3150--3175.
	
	\bibitem{Beck1993}
	J. Beck,
	Achievement games and the probabilistic method,
	in: {\it Combinatorics, Paul Erd\H{o}s is Eighty}, Bolyai Soc. Math. Stud. {\bf 1} (1993), 51--78.
	
	\bibitem{Bednarska2024}
	M. Bednarska-Bzd\k{e}ga,
	Off-diagonal online size Ramsey numbers for paths,
	{\it European J. Combin.} {\bf 118} (2024), 103873.
	
	\bibitem{Conlon2010}
	D. Conlon,
	On-line Ramsey numbers,
	{\it SIAM J. Discrete Math.} {\bf 23} (2010), 1954--1963.
	
\bibitem{Conlon2019}
D. Conlon, J. Fox, A. Grinshpun, and X. He,
Online Ramsey numbers and the subgraph query problem,
in \emph{Building Bridges II},
Springer, 2019, 159--194.
	
	
	\bibitem{Cyman2015}
	J. Cyman, T. Dzido, J. Lapinskas, and A. Lo,
	On-line Ramsey numbers of paths and cycles,
	{\it Electron. J. Combin.} {\bf 22} (2015), \#P1.15.
	
	\bibitem{Dybizbanski2020}
	J. Dybizba\'nski, T. Dzido, and R. Zakrzewska,
	On-line Ramsey numbers for paths and short cycles,
	{\it Discrete Appl. Math.} {\bf 282} (2020), 265--270.
	
	\bibitem{Grytczuk2008}
	J. A. Grytczuk, H. A. Kierstead, and P. Pra\l{}at,
	On-line Ramsey numbers for paths and stars,
	{\it Discrete Math. Theor. Comput. Sci.} {\bf 10} (2008), 63--74.
	
	\bibitem{GordinowiczPralat2008}
	P. Gordinowicz and P. Pra\l{}at,
	Small on-line Ramsey numbers -- a new approach,
	{\it Contrib. Discrete Math.} {\bf 13} (2018), 101--111.
	
	\bibitem{Kurek2005}
	A. Kurek and A. Ruci\'nski,
	Two variants of the size Ramsey number,
	{\it Discuss. Math. Graph Theory} {\bf 25} (2005), 141--149.
	
	\bibitem{Latip2021}
	F. N. N. B. Mohd Latip and T. S. Tan,
	A note on on-line Ramsey numbers of stars and paths,
	{\it Bull. Malays. Math. Sci. Soc.} {\bf 44} (2021), 3511--3521.
	
	\bibitem{Mond2024}
	A. Mond and J. Portier,
	The asymptotic of off-diagonal online Ramsey numbers for paths,
	{\it European J. Combin.} {\bf 122} (2024), 104032.
	
	\bibitem{Pralat2008}
	P. Pra\l{}at,
	A note on small on-line Ramsey numbers for paths and their generalization,
	{\it Australas. J. Combin.} {\bf 40} (2008), 27--36.
	
	\bibitem{Pralat2008cliques}
	P. Pra\l{}at,
	$\overline{\mathcal R}(3,4)=17$,
	{\it Electron. J. Combin.} {\bf 15} (2008), \#R67.
	
	\bibitem{Pralat2012}
	P. Pra\l{}at,
	A note on off-diagonal small on-line Ramsey numbers for paths,
	{\it Ars Combin.} {\bf 107} (2012), 295--306.
	
	\bibitem{Song2025}
	R. Song, S. Wang, and Y. Zhang,
	Online Ramsey numbers of $K_{1,3}$ versus paths,
	{\it Discrete Appl. Math.} {\bf 377} (2025), 218--224.
	
	\bibitem{Zhang2023}
	Y.B. Zhang and Y.X. Zhang,
	Proof of a conjecture on online Ramsey numbers of paths,
	{\it arXiv:2302.13640} (2023).% \url{https://doi.org/10.48550/arXiv.2302.13640}.
\end{thebibliography}
\end{document}